\documentclass[a4paper,12pt]{amsart}

\usepackage{amsfonts}
\usepackage{amsmath}
\usepackage{amssymb}

\usepackage{mathrsfs}
\numberwithin{equation}{section}
\theoremstyle{plain}
\newtheorem{thm}{Theorem}[section]

\newtheorem{cor}[thm]{Corollary}
\newtheorem{lem}[thm]{Lemma}

\theoremstyle{definition}

\newcommand{\be}{\begin{equation}}
\newcommand{\ee}{\end{equation}}

\def\R{{\mathbb R}}

\def\S3{{{\mathbb S}^3}}
\def\SU2{{{\rm SU}(2)}}
\def\Rn{{{\mathbb R}^n}}

\def\p#1{{\left({#1}\right)}}

\def\abs#1{{\left|{#1}\right|}}

\def\Rn{{\mathbb R}^n}
\def\R2n{{\mathbb R}^{2n}}

\def\S{{\mathcal S}}

\def\Rn{{\mathbb R}^n}

\def\R2{{\mathbb R}^2}
\def\R2n{{\mathbb R}^{2n}}

\def\S{{\mathcal S}}

\begin{document}

\title[On global inversion of homogeneous maps]
{On global inversion of homogeneous maps}

\author[]{Michael Ruzhansky and Mitsuru Sugimoto}
\address{
  Michael Ruzhansky:
  \endgraf
  Department of Mathematics
  \endgraf
  Imperial College London
  \endgraf
  180 Queen's Gate, London SW7 2AZ, UK
  \endgraf
  {\it E-mail address} {\rm m.ruzhansky@imperial.ac.uk}
  \endgraf
  \medskip
  Mitsuru Sugimoto:
  \endgraf
  Graduate School of Mathematics
  \endgraf
  Nagoya University
  \endgraf
  Furocho, Chikusa-ku, Nagoya 464-8602, Japan
  \endgraf
  {\it E-mail address} {\rm sugimoto@math.nagoya-u.ac.jp}
  }
  
\thanks{The first author was supported by the EPSRC
 Leadership Fellowship EP/G007233/1 and by EPSRC Grant
 EP/K039407/1}
\date{\today}


\begin{abstract}
In this note we
prove a global inverse function theorem for homogeneous mappings on $\Rn$.
The proof is based on an adaptation of the 
Hadamard's global inverse theorem which provides conditions for a function to be globally
invertible on $\Rn$. For the latter adaptation, we give a short elementary proof assuming a topological result. 
\end{abstract}

\maketitle

\section{Introduction}

The purpose of this note is to prove a global inverse function for homogeneous 
mappings on $\Rn$. The main difficulty is in the fact that such mappings are not
smooth at the origin and thus, the known global inverse function
theorems on $\Rn$ are
not readily applicable.
The main motivation for studying such inverses are applications
to the global invertibility of Hamiltonian flows, and further applications to
construction of suitable phase functions of Fourier integral operators, 
a topic that will be addressed elsewhere.

Thus, our aim here is to establish the existence and the following properties for the
global inverses of homogeneous mappings. 

\begin{thm}\label{COR-APP:global-inverse-hom}
Let $n\geq 3$ and $1\leq k\leq\infty$.
Let $f:\Rn\backslash 0\to\Rn$ 
be a positively homogeneous mapping of order 
$\varkappa>0$, i.e.
$$
f(\tau\xi)=\tau^\varkappa f(\xi) \textrm{ for all } \tau>0,
\ \xi\not=0.
$$
Assume that $f\in C^k(\Rn\backslash 0)$ and
that its Jacobian never vanishes on $\Rn\backslash 0$. 
Then $f$ is bijective from $\Rn\backslash 0$ to 
$\Rn\backslash 0$, its global inverse satisfies
$f^{-1}\in C^k(\Rn\backslash 0)$ and is positively 
homogeneous of order $1/\varkappa$.
Moreover, if we extend $f$ to $\Rn$ by setting $f(0)=0$,
the extension is a global homeomorphism on $\Rn$.
\end{thm} 

We remark that for $n=2$ the conclusion is not always true
because, for example, the Jacobian of $f(x,y)=(x^2-y^2,2xy)$ never vanishes
on $\mathbb R^{2}\backslash 0$ but $f$ is not globally invertible since
$f(x,y)=f(-x,-y)$.

\medskip
Our proof is based on the application of an adaptation of the Hadamard
global inverse theorem. Thus, the second aim of this note is to give a short elementary proof
for it assuming a well-known topological result (very likely also known to Hadamard).

\medskip
This kind of global inverse function theorem is a classical subject 
and of independent interest, but sometimes it also plays an important role 
when we discuss the global $L^2$ boundedness of oscillatory integrals.
For example, Asada-Fujiwara \cite{AF} established this boundedness
based on the global invertibility of maps defined by
phase functions which are smooth everywhere.
Generalising such results to the case of homogeneous phase as in Theorem
\ref{COR-APP:global-inverse-hom} is also 
important but it is not straightforward because of the 
singularity at the origin.
The global $L^{2}$ boundedness of oscillatory integrals with phases as in 
Theorem \ref{COR-APP:global-inverse-hom}
has been analysed by the authors in \cite{RS-CPDEs} and applied to 
questions of global smoothing for partial differential equations in \cite{RS-MA}.
The application of Theorem \ref{COR-APP:global-inverse-hom} to the global 
analysis of hyperbolic partial differential equations  and the corresponding
Hamiltonian flows will appear elsewhere.

\section{Global inverses}

%The topic of inverse function theorems has a long history which we will briefly 
%review below.

Let us start with a topological result we assume.
A differentiable map between manifolds is called a $C^{1}$-diffeomorphism if
it is one-to-one and its inverse is also differentiable.
A mapping $f$ is called proper if $f^{-1}(K)$
is compact whenever $K$ is compact.

\begin{thm}\label{PROP-APP:NR}
Let $M$ and $N$ be connected, oriented, $d$-dimensional
$C^1$-manifolds, without boundary.
Let $f:M\to N$ be a proper $C^1$-map
such that the Jacobian $J(f)$ never vanishes.
Then $f$ is surjective.
If $N$ is simply connected in addition, then
$f$ is also injective.
\end{thm} 

This fact was known to Hadamard, 
but a rigorous proof for surjectivity can be found
in \cite{Nijenhuis-Richardson:Jacobians-1962}.
As for the injectivity, a precise proof
can be found in \cite[Section 3]{Gordon:inverse-function-1972}).
We remark that it follows 
in principle from the general fact that
a proper submersion between smooth
manifolds without boundary is a fibre bundle, meaning
in our setting that it is a covering map.
Since a simply connected manifold is its own
universal covering space, it implies 
that $f$ is a diffeomorphism.

We will mostly discuss 
the mappings $f:\Rn\to\Rn$,
in which case we denote its Jacobian by $J(f):=\det Df:=\det (\partial f_{i}/\partial x_{j})$.
The Hadamard global
inverse function theorem states:

\begin{thm}\label{Hadamard2}
A $C^1$-map $f:\Rn\to\Rn$ is a $C^1$-diffeomorphism
if and only if 
the Jacobian $\det Df(w)$ never vanishes and
$|f(y)|\to\infty$ whenever $|y|\to\infty$.
\end{thm}

This theorem goes back to Hadamard
\cite{Hadamard:planes-CRAS-1906, Hadamard:Ponctuelles-BSMF-1906,Hadamard:Oeuvres}.
In fact, in 1972 W. B. Gordon wrote
``{\em This theorem goes back at least to Hadamard, but it does not appear to be
`well-known'. Indeed, I have found that most people do not believe it when
they see it and that the skepticism of some persists until they see two proofs.}''
The reason behind this is that while we know that the function is locally a $C^{1}$-diffeomorphism
by the usual local inverse function theorem, 
the condition that $|f(y)|\to\infty$ as $|y|\to\infty$, guarantees that the function is
both {\em injective} and, more importantly, {\em surjective} on the whole of $\Rn$.
And indeed, W. B. Gordon proceeds in \cite{Gordon:inverse-function-1972}
by giving two different proofs for it, for
$C^{2}$ and for $C^{1}$ mappings.

\medskip
It turns out that for dimensions $n\geq 3$
the conclusion remains still valid even if we
relax the assumption, in some sense rather 
substantially, almost removing it at a finite number of points. 
While often this is not the case (the famous one being the 
`hairy ball theorem', which fails completely if we assume that the vector field
may be not differentiable at one point), the theory of covering spaces
assures that it is the case here.
In fact, here we have the following:

\begin{thm}\label{THM-APP:global-inverse}
Let $n\geq 3$. Let $a\in\Rn$.
Let $f:\Rn\to\Rn$ be such that 
$f$ is $C^1$ on $\Rn\backslash \{a\}$, with
$\det Df\not=0$ on $\Rn\backslash \{a\}$, and that $f$ is
continuous at $a$. 
Let $b:=f(a)$, and assume that 
$f(\Rn\backslash \{a\})\subset \Rn\backslash \{b\}$
and that $|y|\to\infty$ implies $|f(y)|\to\infty$.
Then 
the mapping $f:\Rn\to\Rn$ is a global homeomorphism and its restriction
$f:\Rn\backslash \{a\}\to \Rn\backslash \{b\}$
is a global $C^1$-diffeomorphism.
\end{thm} 

%We note that if we assume that $f$ is also differentiable at $a$, Theorem
%\ref{THM-APP:global-inverse} together with the local inverse function
%theorem applied at $a$ imply Theorem \ref{Hadamard2}.

%We note that Theorem \ref{THM-APP:global-inverse} follows 
%in principle from the general fact that
%a proper submersion between smooth
%manifolds without boundary is a fibre bundle, meaning
%in our setting that it is a covering map.
%Since $\mathbb R \times {\mathbb S}^{n-1}$ for $n\geq 3$
%is simply connected it implies 
%that $f$ is diffeomorphism.

While general topological considerations as outlined above are possible here,
we prefer to also give an elementary proof in Section \ref{SEC:proofs} of the fact that
Theorem \ref{PROP-APP:NR} implies Theorem \ref{THM-APP:global-inverse},
also noting that exactly the same proof yields the following further
version:

\begin{thm}\label{THM-APP:global-inverse2}
Let $n\geq 3$. 
Let $A\subset\Rn$ be a closed set. 
Let $f:\Rn\to\Rn$ be such that 
$f$ is $C^1$ on $\Rn\backslash A$, with
$\det Df\not=0$ on $\Rn\backslash A$, that $f$ is
continuous and injective on $A$, and that
$\Rn\backslash f(A)$ is simply connected.
Assume that 
$f(\Rn\backslash A)
\subset \Rn\backslash f(A)$ and that
$|y|\to\infty$ implies $|f(y)|\to\infty$. 
Then 
the mapping $f:\Rn\to\Rn$ is a global homeomorphism and its restriction
$f:\Rn\backslash A\to \Rn\backslash f(A)$
is a global $C^1$-diffeomorphism.
\end{thm} 

Given Theorem \ref{THM-APP:global-inverse}, we can apply it to derive 
Theorem \ref{COR-APP:global-inverse-hom}:

\begin{proof}[Proof of Theorem \ref{COR-APP:global-inverse-hom}]
First, let us extend $f$ to $\Rn$ by setting $f(0)=0$ and show
that $f$ is continuous at $0$.
We observe that $f({\mathbb S}^{n-1})$ has a finite maximum.
Let $\xi_j\to 0$, and $\xi_j\not=0$ for all $j$. Then
$$
 |f(\xi_j)|=|\xi_j|^\varkappa 
 \abs{f\p{\frac{\xi_j}{|\xi_j|}}}\leq C|\xi_j|^\varkappa\to 0,
$$
so that $f$ is continuous at $0$.

Let us now check that other conditions of Theorem \ref{THM-APP:global-inverse}
are satisfied.
We observe that $f({\mathbb S}^{n-1})$ has a positive minimum
$\min_{|\xi|=1} |f(\xi)|= c_0>0$.
Indeed, if $f(\omega)=0$
for some $\omega\in{\mathbb S}^{n-1}$,
then $f(t\omega)(=t^\kappa f(\omega))=0$ for any $t>0$.
Differentiating it in $t$, we have $\omega=0$
since the Jacobian of $f$ never vanishes on $\Rn\backslash 0$,
which is a contradiction. 
Then we have
$$
 |f(\xi)|=|\xi|^\varkappa 
 \abs{f\p{\frac{\xi}{|\xi|}}}\geq c_0|\xi|^\varkappa,\quad \xi\neq0,
$$
which induces that $f(\Rn\backslash 0)\subset \Rn\backslash 0$
and that $|y|\to\infty$ implies $|f(y)|\to\infty$.
Therefore, by Theorem \ref{THM-APP:global-inverse},
$f:\Rn\to\Rn$ is a homeomorphism, and
$f^{-1}$ is $C^k$ on $\Rn\backslash 0$ by the usual local inverse
function theorem. Let us finally show that $f^{-1}$ is positively
homogeneous of order $1/\varkappa$. Indeed, for every
$\tau>0$ and $\xi\not=0$ we have
$f^{-1}(\tau^\varkappa f(\xi))=f^{-1}(f(\tau\xi))=\tau\xi$.
Since $f$ is invertible, $\eta=f(\xi)\not=0$, and we have
$f^{-1}(\tau^\varkappa \eta)=\tau f^{-1}(\eta)$, or
$f^{-1}(\tau \eta)=\tau^{1/\varkappa} f^{-1}(\eta)$. 
%Finally, since $1/\varkappa>0$, by the same argument as in
%the beginning of the proof, we get that $f^{-1}$ is
%continuous at $0$. Since it is also continuous on $\Rn\backslash 0$,
%$f$ is a homeomorphism on $\Rn$.
\end{proof}

\section{Proofs}
\label{SEC:proofs}

First we observe that 
in the setting of Theorem \ref{THM-APP:global-inverse},
by translation (by $a$) in $x$ and
by subtracting $b$ from $f$, we may assume without
loss of generality that $a=b=0$.
To prove Theorem \ref{THM-APP:global-inverse},
we start with preliminary statements.

\begin{lem}\label{LEM-APP:compacts}
Let $F\subset\Rn\backslash 0$. Then $F$ is compact
in $\Rn\backslash 0$ if and only if it is compact in $\Rn$.
\end{lem} 
\begin{proof}
Assume that $F\subset \Rn\backslash0$ is compact in
$\Rn\backslash 0$. Let $F\subset \bigcup_\alpha V_\alpha$ for a 
family of sets $V_\alpha$ which are open in $\Rn$. Then
$$
 F\subset \p{\bigcup_\alpha V_\alpha}\cap (\Rn\backslash 0)
 =\bigcup_\alpha \p{V_\alpha\cap (\Rn\backslash 0)},
$$
so that $F$ is covered by a family of sets 
$V_\alpha\cap (\Rn\backslash 0)$ which are open in $\Rn\backslash 0$. 
Since $F$ is compact in
$\Rn\backslash 0$, there is a finite subfamily $V_j$, $j=1,\cdots,m$,
such that 
$$F\subset  \bigcup_{j=1}^m \p{V_j \cap (\Rn\backslash 0)}
 =\p{\bigcup_{j=1}^m V_j}\cap (\Rn\backslash 0).$$
Hence $F\subset \bigcup_{j=1}^m V_j$, so that $F$ is compact in
$\Rn$.

Conversely, assume that $F\subset\Rn\backslash 0$ is compact in
$\Rn$, and let 
$F\subset \bigcup_\alpha U_\alpha$, for a family of sets
$U_\alpha\subset\Rn\backslash 0$ which are open in 
$\Rn\backslash 0$. Then
$U_\alpha=V_\alpha \cap (\Rn\backslash 0)$, for 
some $V_\alpha$ open in $\Rn$. Hence $F\subset \bigcup_\alpha V_\alpha$,
and by compactness of $F$ in $\Rn$, there is a finite subcovering
$F\subset \bigcup_{j=1}^m V_j$.
Since $F\subset\Rn\backslash 0$, we have
$$
F\subset \p{\bigcup_{j=1}^m V_j}\cap (\Rn\backslash 0)=
\bigcup_{j=1}^m \p{V_j \cap (\Rn\backslash 0)}=
\bigcup_{j=1}^m U_j,
$$
which proves that $F$ is compact in $\Rn\backslash 0$.
\end{proof} 

We recall that a mapping $f$ is called proper if $f^{-1}(K)$
is compact whenever $K$ is compact.

\begin{lem}\label{COR-APP:compacts}
Let $f:\Rn\to\Rn$ be proper and such that
$f(0)=0$ and $f(\Rn\backslash 0)\subset \Rn\backslash 0$.
Then the restriction $f:\Rn\backslash 0\to \Rn\backslash 0$
is proper.
\end{lem} 
\begin{proof}
Let $K\subset \Rn\backslash 0$ be compact in
$\Rn\backslash 0$. By Lemma \ref{LEM-APP:compacts} it is
compact in $\Rn$, and, since $f$ is proper, the set
$f^{-1}(K)$ is compact in $\Rn$. We notice that
if $0\in f^{-1}(K)$ then we would have $0=f(0)\in K$, which
is impossible since $K\subset \Rn\backslash 0$. Hence
$f^{-1}(K)\subset \Rn\backslash 0$, and by
Lemma \ref{LEM-APP:compacts} again, the set
$f^{-1}(K)$ is compact in $\Rn\backslash 0$. Hence
the restriction $f:\Rn\backslash 0\to \Rn\backslash 0$
is proper.
\end{proof}

\begin{lem}\label{LEM-APP:Rn-proper}
Let $f:\Rn\to\Rn$ be continuous everywhere. Then
$f$ is proper if and only if
$|y|\to\infty$ implies $|f(y)|\to\infty$.
\end{lem} 
\begin{proof}
We show the if part.
Let $K\subset \Rn$ be compact. Then it is closed and hence
$f^{-1}(K)\subset\Rn$ is closed. Suppose 
$f^{-1}(K)$ is not bounded. Then there is a sequence
$y_j\in f^{-1}(K)$ such that $|y_j|\to\infty$. Hence
$f(y_j)\in K$ and also $|f(y_j)|\to\infty$ by the assumption
on $f$, which yields a contradiction with the boundedness
of $K$. The converse implication is clearly also true.
\end{proof} 

\begin{lem}\label{LEM-APP:Rn-homeo}
Let $f:\Rn\to\Rn$ be proper, injective, and continuous.
Then $f$ is an open map.
\end{lem} 
\begin{proof}
Let us assume that $f(U)$ is not open
for an open subset $U\subset\Rn$.
Then there is a point $a\in U$ such that $f(a)$ is on the boundary
of $f(U)$, and we can construct a sequence $y_j\in\Rn\backslash U$
such that $f(y_j)\to f(a)$.
Since $f$ is proper, there exists a subsequence $y_j'$
which converges to some point $b\in\Rn\backslash U$.
Note that $b\not=a$.
Since $f$ is continuous, $f(y_j')\to f(b)$,
but we also have $f(y_j')\to f(a)$ which contradicts to the
fact that $f$ is injective.
\end{proof}

The following result is a straight forward consequence of Theorem
\ref{PROP-APP:NR}.

\begin{cor}\label{PROP-APP:onto}
Let $n\geq 2$.
Let $f:\Rn\to\Rn$ be proper and such that 
and $f(0)=0$ and
$f(\Rn\backslash 0)\subset \Rn\backslash 0$.
Moreover, assume that $f$ is $C^1$ on $\Rn\backslash 0$, with
$\det Df\not=0$ on $\Rn\backslash 0$. Then 
the restriction $f:\Rn\backslash 0\to \Rn\backslash 0$
is surjective.
If $n\geq3$ in addition,
$f:\Rn\backslash 0\to \Rn\backslash 0$
is also injective.
\end{cor} 
\begin{proof}
By Lemma \ref{COR-APP:compacts}, the restriction of $f$
is a proper map from $M=\Rn\backslash 0$ to $N=\Rn\backslash 0$.
Note that $N$ is simply connected if $n\geq3$.
Then Theorem \ref{PROP-APP:NR} implies the statement.
\end{proof}

With all these facts, Theorem  \ref{THM-APP:global-inverse} is immediate:

\begin{proof}[Proof of Theorem \ref{THM-APP:global-inverse}]
By Corollary \ref{PROP-APP:onto} and Lemma \ref{LEM-APP:Rn-proper},
the map $f:\Rn\backslash 0\to \Rn\backslash 0$ is bijective.
Hence it is a global $C^1$ diffeomorphism by the usual local inverse
function theorem.
Furthermore, the map $f:\Rn\to \Rn$ is also bijective
since $f(0)=0$, hence the global inverse $f^{-1}:\Rn\to \Rn$ exists
and is continuous by Lemma \ref{LEM-APP:Rn-homeo}.
Since $f:\Rn\to \Rn$ is also continuous, it is a homeomorphism.
\end{proof}

\par
\bigskip
\par\noindent
{\it Acknowledgement.} The authors would like to thank
Professor Adi Adimurthi for valuable remarks on the first version of
our manuscript, leading to its considerable improvement.
%and for the correspondence concerning Theorem \ref{THM-APP:global-inverse2}. 

%\bibliographystyle{alphaabbr}
%\bibliography{bib_pdo-12-12-1}
%\end{document}

\end{document}